\documentclass[12pt,a4]{article}
\usepackage{epsfig}
\usepackage{psfrag}
\usepackage{amsbsy}
\usepackage{latexsym}
\usepackage[mathscr]{eucal}
\usepackage{amsfonts,amsmath,amsthm}
\usepackage{amssymb}

\newcounter{rem}
\newtheorem{lemma}{Lemma}[section]
\newtheorem{prop}[lemma]{Proposition}

\newtheorem{theorem}[lemma]{Theorem}

\newenvironment{rem}{\addtocounter{rem}{1}\par \noindent{\bf Remark \arabic{section}.\arabic{rem}} \sf}{\par}

\newcommand{\beqn}{\begin{equation}}
\newcommand{\eeqn}{\end{equation}}

\def\N{\mathbb{N}}

\def\span{\text{span\,}}

\def\cA{{\cal A}}

\def\cG{{\cal G}}

\def\cX{{\cal X}}

\def\<{\langle}

\def\>{\rangle}

\def\<{\langle}
\def\>{\rangle}

\begin{document}

\title{ On left democracy function}
\author{P. Wojtaszczyk  \thanks{The author was partially supported by the ``HPC
Infrastructure for Grand Challenges of Science and Engineering”
Project, co-financed by the European Regional Development Fund
under the Innovative Economy Operational Programme" and Polish NCN
grant DEC2011/03/B/ST1/04902.}
  }

\date{}

\maketitle

{ \large \sl \hfill To Lech Drewnowski, with thanks 

\hfill for many years of nice mathematics}

\section{Introduction}
The aim of this note is to settle some problems left open in \cite{GHN}. 
Suppose we have a Banach  space $X$ with a normalised basis $(x_n)_{n=1}^\infty$. For $x=\sum_{n=1}^\infty a_nx_n\in X$ and $N=1,2,\dots$ 
we define a non-linear operator 
\begin{equation} \label{op}
 \cG_n(x)=\sum_{n\in \Lambda_N} a_n x_n
\end{equation}
where $\Lambda_N$ is any $N$-element subset of indices such that $\min_{n\in \Lambda_N} |a_n|\geq \max_{n\notin \Lambda_N} |a_n|$. Note
that the set $\Lambda_N$ may not be uniquely defined; in such a  case we are allowed to take arbitrary choice.
This is a theoretical model of many practically important tresholding operators. Systematic study of such operators was undertaken in the
last years of the XX century (see e.g.
\cite{T1,KT,W1}) and is an active area of research. It became apparent already in  \cite{W1} that quantities like $\|\sum_{n\in A} x_n\|$
are important for the
properties of this operator. The basis is called {\em democratic} \cite{KT} if those quantities depend essentially only on number of
elements
of $A$, more precisely if there exists a constant $C$ such that for all sets $A,B$ with $\#A=\#B$ we have
\begin{equation}\label{KT10}
 \|\sum_{n\in A} x_n\|\leq C\|\sum_{n\in B} x_n\|
\end{equation}
The main result of \cite{KT} asserts that a basis is unconditional and democratic if and only if it is {\em greedy} what means that
$\cG_N(x)$ is (up to a constant) a best $N$--term approximation of $x$ by elemets
$\{x_n\}_{n=1}^\infty$; more precisely there exists a constant $C$ such that for all $x\in X$ and $N=1,2,\dots $ we have  $\| x- \cG_N(x) \|
< C \sigma_N(x)$, ($\sigma_N$ is defined in (\ref{sigma})). 

A more detailed study resulted in the definition \cite{KaT} of the  left democracy function 
\begin{equation} \label{LD}
 h_l(N)=\inf_{\#\Lambda=N}\left\|\sum_{n\in \Lambda_N} x_n\right\|
\end{equation}
and right democracy function
\begin{equation} \label{RD}
 h_r(N)=\sup_{\#\Lambda=N}\left\|\sum_{n\in \Lambda_N} x_n\right\|.
\end{equation}
The detailed study of the role of those functions in approximation properties of the basis $(x_n)_{n=1}^\infty$ was recently undertaken in
\cite{GHN}

In the rest of this note we will always assume that $(x_n)_{n=1}^\infty$ is a lattice unconditional basis i.e. 
\begin{equation}\label{uncond}
 \left\|\sum_{n=1}^\infty \lambda_na_n x_n\right\|\leq  \left\|\sum_{n=1}^\infty a_n x_n\right\|
\end{equation}
whenever $|\lambda_n|\leq 1$. Since every  space with an unconditional basis can be renormed so that the basis will satisfy
(\ref{uncond}) we really consider unconditional bases here. We will use standard Banach space conventions and results, c.f. \cite{W2}.

{\bf Acknowledgements:} I would like to express my gratitude to professors  C. Cabrelli, G. Garrig\'os, E. Hernandez and U. Molter
for sharing their ideas with me and for  kind permission to present some of their unpublished results in this paper.

\section{Space with nondoubling left democracy function.}

A function positive $\phi(n)$ defined for $n=1,2,\dots$ is doubling if there exists a $C$ such that $\phi(2n)\leq C\phi(n)$ for all $n$
Such functions appear in many places in analysis. It was observed in \cite[Prop. 2.4]{GHN} that $h_r(N)$ is doubling and that both $h_l$
and $h_r$ are increasing.  The question if $h_l$ is always doubling was left open \cite[Remark 2.5]{GHN} and in some results an assumption
that $h_l$ is doubling appears.

Now we are ready to state one of the main results of this note

\begin{theorem}\label{PW1}
 There exists a Banach space $X$ with the basis $(e_j)_{j=1}^\infty$ (satisfying (\ref{uncond})) such that the left democracy function $h_l$
of this basis is not doubling.
\end{theorem}

We will say that the basis $(x_n)_{n=1}$ is  $1$-symmetric if for every permutation of indices $\pi$ and all sequences $(\epsilon_n)_{n=1}$
of numbers with absolute value one  and all sequences $(a_n)_{n=1}$ of coefficients we have
\begin{equation}
 \left\|\sum_{n=1} a_n x_n\right\| =\left\|\sum_{n=1}\epsilon_n a_n x_{\pi(n)}\right\|.
\end{equation}

For natural numbers $n\leq N$ let $\cX(n,N,2)$ be
a  Banach space with $1$-symmetric basis $(e_\mu)_{\mu=1}^N$ such that
$$\|\sum_{j\in \Gamma} e_j\|=\begin{cases}
                              \sqrt{\#\Gamma}& \mbox{ when } \#\Gamma \leq n\\
                               \sqrt{n}& \mbox{ when } \#\Gamma>n.
                             \end{cases}
$$
One example of such a space can be defined as
$$\|\sum_{j=1}^N x_j e_j\|:=\sup\left\{ \sum_{j\in \Gamma } x_j v_j\ \ \right\}$$
where the supremem is taken over all subsets $\Gamma\subset\{1,\dots, N\}$ of cardinality $\leq n$ and all sequences
$(v_j)_{j\in \Gamma}$ with $\sum_{j\in \Gamma}|v_j|^2\leq 1$. It is easy to see that it is a norm and the norm of a vector is the $\ell_2$
norm of its $n$ biggest (up to absolute value) coefficients. It also immediately follows from the definition that it is $1$-symmetric.

Given an increasing sequence of natural numbers $a_j$ for $j=1,2,\dots $ with $a_1\geq 4$ and $\lim_{j\to
\infty}a_j=\infty$ we define $n_k=\prod_{j=1}^k a_j$. 
This implies $n_{k+1}/n_k\geq4$

Now let us define the space 
$$\cX=:\left(\sum_{k=1}^\infty\cX(n_k,n_{k+1},2)\right)_2.$$
This space has a natural basis $(e_\mu)_{\mu\in Y}$ where $Y=\bigcup_{k=1}^\infty Y_k$ where $\#Y_k=n_{k+1}$ and $\span
(e_\mu)_{\mu\in Y_k}=\cX(n_k,n_{k+1},2)$.

\begin{lemma}
 For the space $\cX$ defined above the function $h_l(n)$ is not doubling. 
\end{lemma}
\begin{proof} We will show that  $\sup_n\frac{h_l(2n)}{h_l(n)}=\infty$. Let us take $\Gamma$ with $\#\Gamma=n_{k+1}$. If $\Gamma=Y_k$  we
get
$\|\sum_{j\in \Gamma}e_j\|=\sqrt{n_k}$ so $h_l(n_{k+1})\leq \sqrt{n_k}$.

Now let us take $\Gamma$ with $\#\Gamma=2n_{k+1}$. We have
\begin{eqnarray}
 \# \bigcup_{j=1}^kY_j &=& n_2+n_3+\dots+n_k+n_{k+1}\\
&\leq& n_{k+1}(1+\frac14+\frac{1}{4^2}+\frac{1}{4^{k-1}})\leq \frac43n_{k+1}.
\end{eqnarray}
This means that at least $\frac23 n_{k+1}$ elements from $\Gamma$  are in 
$ \bigcup_{j=k+1}^\infty Y_j. $
Let $\Gamma^1$ be a fixed set of  such elements with  $\frac23 n_{k+1}\leq \#\Gamma^1\leq n_{k+1} $ and let us write
$\Gamma^1=\bigcup_{s=k+1}^\infty A_s$ where $A_s=\Gamma^1\cap Y_s$. Since each
of $A_s$'s has at most $n_{k+1}$ elements
 we get
\begin{eqnarray}
 \|\sum_{j\in \Gamma} e_j\|&\geq &\|\sum_{j\in \Gamma^1} e_j\|=\sqrt{\sum_{s=k+1}^\infty\|\sum_{j\in A_s}e_j\|^2  }\\
&=&\sqrt{ \sum_{s=k+1}^\infty \# A_s }=\sqrt{\#\Gamma^1}\geq \sqrt{\frac23 n_{k+1}}.
\end{eqnarray}
So $h_l(2n_{k+1})\geq \sqrt{\frac23 n_{k+1}}$ and we get
$$\frac{h_l(2n_{k+1})}{h_l(n_{k+1})}\geq\frac{\sqrt{\frac23 n_{k+1}}}{\sqrt{n_k}}=\sqrt{\frac23}\sqrt{a_{k+1}}  $$
Since $a_k$ tends to infinity we get the claim.
\end{proof}

\begin{rem}
 A more careful analysis should show that $h_l(n)$ is exactly equal to the norm of the sum of the first $n$ unit
vectors.
\end{rem}
\begin{rem}
 Clearly we can use other values of $p$ in place of $2$. 
\end{rem}

\section{Approximation spaces}

It is standard in approximation theory to define spaces of elements which admit some rate of approximation. In our context two spaces are
esential. We define them for a fixed Banach space with the basis $(x_n)_{n=1}^\infty$.
\begin{enumerate}
 \item Non-linear approximation space $\cA^\alpha_q$ with $\alpha>0$ and $0<q<\infty$ defined as
$$\cA_q^\alpha=\left\{ x\in X \ :\ \|x\|_{\cA_q^\alpha} =\|x\|+\left[\sum_{N=1}^\infty(N^\alpha
\sigma_N(x))^q\tfrac1N\right]^{1/q}<\infty\right\}$$ 
and for $q=\infty$ we define
$$\cA_\infty^\alpha=\left\{ x\in X \ :\ \|x\|_{\cA_q^\alpha} =\|x\|+\sup_{N\geq 1}N^\alpha \sigma_N(x)<\infty\right\}$$ 
$\sigma_N(x)$  is the error of the best $N$-term approximation i.e.
\begin{equation}\label{sigma}\sigma_N(x)=\inf\{\|x-\sum_{n\in \Lambda} b_nx_n\|\  :\ \#\Lambda=N  \mbox{  and $b_n$'s are
arbitrary}\}\end{equation}
\item Greedy classes $\cG_q^\alpha$ are defined in the same way but we replace $\sigma_N(x)$ by error of a greedy approximation which is
defined as $\gamma_N(x)=\max\|x-\cG_N(x)\|$. The maximum is taken over all $\cG_N(x)$'s in case it is not uniquely defined.
\end{enumerate}
It is well known that $\cA_q^\alpha$ are quasi-Banach spaces with the quasi-norm $\|.\|_{\cA_q^\alpha}$. For the spaces $\cG_q^\alpha$ the
situation is not so clear--we do not know if it is a linear space. Clearly if the basis is greedy then $\sigma_N\sim \gamma_N$ and the
spaces are equal. Also, since always
$\sigma_N(x)\leq \gamma_N(x)$, we have $\cG_q^\alpha\subset \cA_q^\alpha$.   The problem wether the equality $\cG_q^\alpha=
\cA_q^\alpha$  characterise  greedy bases was considered in \cite{GHN}. Actually it turned out to be quite difficult so the authors
considered the problem of equivalence of quantities $\|x\|_{\cA_q^\alpha}$ and $\|x\|_{\cG_q^\alpha}$. Let us say that {\em greedy
approximation is optimal}\footnote{In \cite{GHN} this notion was expressed as ''the  inclusion  $\cA_q^\alpha\hookrightarrow\cG_q^\alpha
$ does not hold''.} for $\alpha$ and $q$ if there exists a constant $C$ such that for every $x\in{\cA_q^\alpha}$ we have 
$$\|x\|_{\cG_q^\alpha}  \leq C\|x\|_{ \cA_q^\alpha} .$$

The main result of this section is the following
\begin{theorem}\label{optimal}
If $(x_n)$ in unconditional, the following are equivalent
\begin{enumerate}
 \item $(x_n)$ is democratic
\item $\| x- \cG_N(x) \| < C \sigma_N(x)$  for all $x$
\item $\| x \|_{\cG^\alpha_q} < C \|  x \|_{\cA^\alpha_q}$ for all (some) $\alpha,q>0$.
\end{enumerate}
 
\end{theorem}
\begin{rem}
 This Theorem for bases with doubling $h_l$ was proved by  C. Cabrelli, G. Garrigós, E. Hernandez and U. Molter and stated without proof in
a note {\em Added in proof} in \cite{GHN}. Below I present their proof with their kind permission.
\end{rem}

\begin{proof} That for unconditional bases 1. is equivalent to 2. was proved by Konyagin--Temlyakov \cite{KT}. $2.\Rightarrow 3.$ is clear
and was already mentioned above. We will prove that for a non-greedy unconditional basis 3. fails.
 We will distinguish two cases: when $h_l$ is doubling and when $h_l$ is not doubling. To prove the first case we need to recall
 Proposition 7.1 from \cite{GHN} 
\begin{prop}\label{GHN7.1}
Suppose that there exist integers $n_\mu\geq k_\mu\geq 1$ for $\mu=1,2,\dots$ such that 
\begin{equation}\label{7.1}
\lim_{\mu\to \infty}\frac{n_\mu}{k_\mu}=\infty \\\ \mbox{ and }\ \ \ \  \frac{h_r(k_\mu)}{h_l(n_\mu)}\geq
C\left(\frac{n_\mu}{k_\mu}\right)^\alpha
\end{equation}
 for some $C>0$ and $\alpha>0$. Then greedy approximation is not optimal for $\alpha $ and any $q\in(0,\infty]$.
\end{prop}

\begin{lemma}[C. Cabrelli, G. Garrigós, E. Hernandez, U. Molter]\label{CGHM}
 Let $\alpha>0$ and $h_r,h_l:\N\rightarrow(0,\infty)$ be {\em any} two increasing functions such that $h_l$ is doubling and $\limsup_{\mu\to
\infty}\frac{h_r(\mu)}{h_l(\mu)} =\infty$. Then there exists integers  $n_\mu\geq k_\mu\geq 1$ for $\mu=1,2,\dots$ such that (\ref{7.1})
holds.
\end{lemma}
\begin{proof}
 We easily see that there exists an increasing sequence of integers
$\{w_\mu\}_{\mu=1}^\infty$ such that
\begin{equation}\label{CGHM1} \lim_{\mu\to
\infty} h_r(w_\mu)/h_l(w_\mu)=\infty.
\end{equation}
 Given $w_\mu$ we fix an integer $r(\mu)$ such that $2^{r(\mu)-1}\leq w_\mu
<2^{r(\mu)}$. Since $h_l$ is doubling,  for any $M,\mu\in\N$ we have
\begin{equation}\label{CGHM2}
 h_l(w_\mu M)\leq h_l(2^{r(\mu)}M)\leq C^{r(\mu)}h_l(M).
\end{equation}
Using (\ref{CGHM1}) we fix an increasing sequence $(k_\mu)_{\mu=1}^\infty$ such that each $k_\mu$ is some $w_{\mu^\prime}$ such that
\begin{equation}\label{CGHM3}
 \frac{h_r(k_\mu)}{h_l(k_\mu)}\geq C^{r(\mu)}w_\mu^\alpha
\end{equation}
and we define $n_\mu=w_\mu k_\mu$, so the first part of (\ref{7.1}) holds . Using (\ref{CGHM2}) and (\ref{CGHM3}),    we obtain
\begin{equation*}
 \frac{h_r(k_\mu)}{h_l(n_\mu)}= \frac{h_r(k_\mu)}{h_l(w_\mu k_\mu)}\geq \frac{h_r(k_\mu)}{ C^{r(\mu)}h_l( k_\mu)}\geq
w_\mu^\alpha=\left(\frac{n_\mu}{k_\mu}\right)^\alpha
\end{equation*}
\end{proof}
To settle the first case we note that a non-greedy basis with doubling $h_l$ satisfies the assumptions of Lemma \ref{CGHM}
so using Proposition \ref{GHN7.1} we get the claim.

Now let us assume that we have a normalised, $1$--unconditional basis $(e_j)_{j=1}^\infty$ with the function $h_l(n)$ {\em not doubling.}
For each $s$ there
exists $n_s$ such that $h_l(2n_s)\geq (s+1)h_l(n_s)$. For simplicity in what follows we will write $\|S\|=\|\sum_{j\in S} e_j\|$. 
Let us fix a set $M_s$ such that $\#M_s=n_s$ and
$\|M_s\|\geq h_l(n_s)\geq\|M_s\|-\frac{1}{s+1}$. Then for any set $D$ disjoint from $M_s$ with $\#D=n_s$ we have
$$\|D\|+\|M_s\|\geq \|M_s\cup D\|\geq h_l(2n_s)\geq (s+1)h_l(n_s)\geq(s+1)(\|M_s\|-\frac{1}{s+1})$$
so for every such $D$ we have $\|D\|\geq s\|M_s\|-1=s h_l(n_s)-1$.

Note that $h_l$ is unbounded (because bounded is doubling).

Given $M_s$ let us take $r=:\lfloor\sqrt s\rfloor$ disjoint sets $V_j$ also disjoint with $M_s$, such that  $\#\bigcup_{j=1}^{r}V_j=n_s$
each of cardinality $\lfloor n_s/r\rfloor$ or $\lceil n_s/r\rceil$. Denote the set $V_j$ with the biggest $\|V_j\|$ as $V^s$.
Since 
$r\max\|V_j\|\geq \|\bigcup_{j=1}^{r} V_j\|\geq s\|M_s\|-1$ we see that 
\begin{equation}\label{LS1}
\|V^s\|\geq  \frac{s}{r}\|M_s\|-\frac1r.
\end{equation}

Put $x_s=\sum_{V^s}e_j+2\sum_{M_s}e_j$. We have $\|x_s\|\leq \|V^s\|+2\|M_s\|\leq 3\|V^s\|$. 
The number of non-zero coefficients of $x_s$ equals $\#V_s+\#M_s\leq 2\#M_s$.  In what follows  we are only interested
in
$k\leq 2\#M_s$ because for $k>2\#M_s$ we have $\cG_k(x_s)=x_s$ and $\sigma_k(x_s)=0$.

For $k\leq \#M_s$ we have
$$\|x_s-\cG_k(x_s)\|\geq\|V^s\|$$
so for $0<q<\infty$ we have
\begin{eqnarray}
 \|x_s\|_{\cG_q^\alpha}&\geq& \left[\sum_{k=1}^{\#M_s}(k^\alpha \|x_s-\cG_k(x_s)\|)^q\frac{1}{k}\right]^{1/q}\\
&\geq & C\|V^s\|(\#M_s)^\alpha.
\end{eqnarray}

and for $q=\infty$ we have
\begin{equation}\label{LS4}
 \|x_s\|_{\cG_\infty^\alpha}\geq\max_{k\leq \#M_s} k^\alpha\|x-\cG_k(x)\| \geq (\#M_s)^\alpha\|V_s\|.
\end{equation}

On the other hand for $k\geq \#V^s$ using (\ref{LS1}) we have
\begin{equation}\label{LS2}
 \sigma_k(x)\leq 2\|M_s\|\leq 2\frac{r\|V^s\|+1}{s}\leq  3\frac{r\|V^s\|}{s}
\end{equation}
and for $k<\#V^s$ 
\begin{equation}\label{LS3}
 \sigma_k(x)\leq \|x\|\leq 3\|V^s\|.
\end{equation}
Therefore using (\ref{LS3}) and (\ref{LS2}), for $q<\infty$ we have
\begin{eqnarray} \label{LS5}
\|x_s\|_{\cA_q^\alpha}&=& \|x_s\|+\left[\sum_{k=1}^{2\#M_s}(k^\alpha \sigma_k(x_s))^q\frac1k\right]^{1/q} \nonumber\\
&\leq& 
3\|V_s\|+\left[(3\|V^s\|)^q\sum_{k=1}^{\#V^s-1}k^{q\alpha-1}+\left(3\frac{r\|V^s\|}{s}\right)^q 
\sum_{k=\#V^s}^{2\#M_s}k^{q\alpha-1}\right]^{1/q}\nonumber \\
&\leq & 3\|V^s\|+\left[C\|V^s\|^q (\#V^s)^{q\alpha} +C(r/s)^q\|V^s\|^q(\#M_s)^{q\alpha}\right]^{1/q}\nonumber\\
&\leq & C(\#M_s)^\alpha\|V^s\|\left(r^{-q\alpha}+(r/s)^q\right)^{1/q}  \nonumber \\
&\leq&
C\|x_s\|_{\cG_q^\alpha}\left(r^{-q\alpha}+(r/s)^q\right)^{1/q} \nonumber \\
&\leq& C\|x_s\|_{\cG_q^\alpha}\left( s^{-q\alpha/2}+s^{-q/2}\right)^{1/q}.
\end{eqnarray}
Analogously for $q=\infty$ we have
\begin{eqnarray}\label{LS6}
 \|x_s\|_{\cA_\infty^\alpha}&=&\|x_s\|+\sup_{k\geq 1}k^\alpha\sigma_k(x_s)\nonumber \\
&\leq &3\|V^s\|+\max_{k<\#V^s}3k^\alpha \|V^s\|+\max_{\#V_s\leq k\leq 2\#M_s}3k^\alpha\frac{s\|V^s\|}{s} \nonumber \\
&\leq & \|V^s\|\left(3+3(\#V^s)^\alpha +3(2\#M_a)^\alpha r/s\right)\nonumber  \\
&\leq &C \|x_s\|_{\cG_\infty^\alpha} \left((\#M_s)^{-\alpha} +s^{-\alpha/2}+s^{-1/2}\right)
\end{eqnarray}
Since $s$ is arbitrary, from (\ref{LS5}) and (\ref{LS6}) we infer that the greedy approximation is not optimal for any $\alpha$ and $q$ also
in the nondoubling case.                                                                                                        
\end{proof}

\noindent  Interdisciplinary Centre for Mathematical and Computational Modelling,
University of Warsaw, 02-838 Warszawa,
ul. Prosta 69, Poland,\\
and
\\
Institut of Mathematics, Polish Academy of Sciences\\00-956 Warszawa, ul. \`Sniadeckich 8, Poland   \\   email: {\tt
wojtaszczyk@mimuw.edu.pl}
\end{document}